\documentclass[12pt]{article}
\usepackage{amssymb}
\usepackage{amsmath}
\usepackage{graphs}
\usepackage{amssymb}
\usepackage{amsthm}
\usepackage{url}
\usepackage[british]{babel}

\usepackage{psgo,array,url}

\theoremstyle{plain}
\newtheorem{definition}{Definition}
\newtheorem{lemma}[definition]{Lemma}
\newtheorem{theorem}[definition]{Theorem}
\newtheorem{conjecture}[definition]{Conjecture}

\newcommand{\lef}{\mathcal{L}}
\newcommand{\pre}{\mathcal{P}}
\newcommand{\n}{\mathcal{N}}
\newcommand{\ri}{\mathcal{R}}
\newcommand{\ti}{\mathcal{T}} 

\newcommand{\gfr}{G_F^{SR}}
\newcommand{\gfl}{G_F^{SL}}
\newcommand{\plus}{+_{\ell}}

\newcommand{\G}{\mathcal{G}}

\begin{document}
\title{Impartial Scoring Play Games}
\author{Fraser Stewart\\
\texttt{fraseridstewart@gmail.com}}
\date{}

\maketitle{}

\centerline{\bf Abstract}
In this paper we will be examining impartial scoring play games.  We first give the basic definitions for what impartial scoring play games are and look at their general structure under the disjunctive sum.  We will then examine the game of nim and all octal games, and define a function that can help us analyse these games.  We will finish by looking at the properties this function has and give many conjectures about the behaviour this function exhibits.

\section{Introduction}

Nim is a game that has been studied by combinatorial game theorists for many years.  In fact many people believe that nim is in fact the only true combinatorial game.  For normal and mis\`ere games an impartial game is defined to be a game where both players have the same options at all stages in the game.  Nim, and all games based on it, are used as the core games in any analysis of impartial games.  The rules of nim are defined as follows;

\begin{enumerate}
\item{The game is played on heaps of beans.}
\item{A player may remove as many beans as he wants from any one heap.}
\item{Under normal play rules the last player to move wins, under mis\`ere play rules the last player to move loses.}
\end{enumerate}

These basic rules in fact led to a whole class of games where the only variations were the rules governing how many beans can be removed from particular heaps.  Sprague and Grundy devised an entire theory that tells you how to find winning strategies for any game of nim played under normal play rules \cite{PG}, \cite{PS}, \cite{RS}, that has become known as Sprague-Grundy theory.

It was not until very recently that a similar theory was devised for nim-games under mis\`ere rules, this theory was devised by Thane Plambeck \cite{TP}.

\subsection{Scoring Play Theory}

Until very recently scoring play games have not received the kind of treatment or analysis that normal and mis\`ere play games have.  This is even stranger when we consider the fact that one of the most well studied and well known scoring play games is the ancient Chinese game of Go.  The general definition of a scoring play game is given below, for further reading on the general structure of scoring play games see \cite{ES} and \cite{FS};

Mathematically scoring games are defined in the following way;

\begin{definition} A scoring play game $G=\{G^L|G^S|G^R\}$, where $G^L$ and $G^R$ are sets of games and $G^S\in\mathbb{R}$, the base case for the recursion is any game $G$ where $G^L=G^R=\emptyset$.\\

\noindent $G^L=\{\hbox{All games that Left can move to from } G\}$\\
\noindent $G^R=\{\hbox{All games that Right can move to from } G\}$,\\

and for all $G$ there is an $S=(P,Q)$ where $P$ and $Q$ are the number of points that Left and Right have on $G$ respectively.  Then $G^S=P-Q$, and for all $g^L\in G^L$, $g^R\in G^R$, there is a $p^L,p^R\in\mathbb{R}$ such that $g^{LS}=G^S+p^L$ and $g^{RS}=G^S+p^R$.

\end{definition}

A concept we will be using throughout this paper is the game tree of a game.  While it may be intuitively obvious to the reader, non-the-less, we feel it is important to define it mathematically.

\begin{definition}
The game tree of a scoring play game $G=\{G^L|G^S|G^R\}$ is a tree with a root node, and every node has children either on the Left or the Right, which are the Left and Right options of $G$ respectively.  All nodes are numbered, and are the scores of the game $G$ and all of its options.
\end{definition}

We also need to define a concept that we call the ``final score''.  This is something which hopefully the reader finds relatively intuitive.  When the game ends, which it will after a finite amount of time, the score is going to determine whether a player won, lost or tied.  Mathematically it is defined as follows.

\begin{definition}  We define the following:

\begin{itemize}
\item{$\gfl$ is called the Left final score, and is the maximum score --when Left moves first on $G$-- at a terminal position on the game tree of $G$, if both Left and Right play perfectly.}
\item{$\gfr$ is called the Right final score, and is the minimum score --when Right moves first on $G$-- at a terminal position on the game tree of $G$, if both Left and Right play perfectly.}
\end{itemize}

\end{definition}

For scoring play the disjunctive sum needs to be defined a little differently, because in scoring games when we combine them together we have to sum the games and the scores separately.  For this reason we will be using two symbols $\plus$ and $+$.  The $\ell$ in the subscript stands for ``long rule'', this comes from \cite{ONAG}, and means that the game ends when a player cannot move on any component on his turn.

\begin{definition}  The disjunctive sum is defined as follows:
$$G\plus H=\{G^L\plus H,G\plus H^L|G^S+H^S|G^R\plus H,G\plus H^R\},$$
\noindent where $G^S+H^S$ is the normal addition of two real numbers.
\end{definition}

The outcome classes also need to be redefined to take into account the fact that a game can end with a tied score.  So we have the following two definitions.

\begin{definition}
\item{$L_>=\{G|\gfl>0\}$, $L_<=\{G|\gfl<0\}$, $L_= =\{G|\gfl=0\}$.}
\item{$R_>=\{G|\gfr>0\}$, $R_<=\{G|\gfr<0\}$, $R_= =\{G|\gfr=0\}$.}
\item{$L_\geq=L_>\cup L_=$, $L_\leq = L_<\cup L_=$.}
\item{$R_\geq=R_>\cup R_=$, $L_\leq = R_<\cup R_=$.}
\end{definition}

\begin{definition}
The outcome classes of scoring games are defined as follows:

\begin{itemize}
\item{$\lef=(L_>\cap R_>)\cup(L_>\cap R_=)\cup(L_=\cap R_>)$}
\item{$\ri=(L_<\cap R_<)\cup(L_<\cap R_=)\cup(L_=\cap R_<)$}
\item{$\n=L_>\cap R_<$}
\item{$\pre=L_<\cap R_>$}
\item{$\ti=L_=\cap R_=$}
\end{itemize}
\end{definition}

\section{Impartial Scoring Games}

The definition of an impartial scoring play game is less intuitive than for normal and mis\`ere play games.  The reason for this is because we have to take into account the score, for example, consider the game $G=\{4|3|2\}$.  On the surface the game does not appear to fall into the category of an impartial game, since Left wins moving first or second, however this game is impartial since both players move and gain a single point, i.e. they both have the same options.

So we will use the following definition for an impartial game;

\begin{definition}
A scoring game $G$ is impartial if it satisfies the following;

\begin{enumerate}
\item{$G^L=\emptyset$ if and only if $G^R=\emptyset$.}
\item{If $G^L\neq \emptyset$ then for all $g^L\in G^L$ there is a $g^R\in G^R$ such that $g^L\plus -G^S=-(g^R\plus -G^S)$.}
\end{enumerate}
\end{definition} 

An example of an impartial game is shown in figure \ref{ex1}.  This game satisfies the definition since $2\plus -3=-(4\plus -3)$, $\{11|4|-3\}\plus -3=-(\{9|2|-5\}\plus -3)=\{8|1|-6\}$ and $11\plus -4=-(-3\plus -4)=9\plus -2=-(-5\plus -2)=7$.  

\begin{figure}[htb]
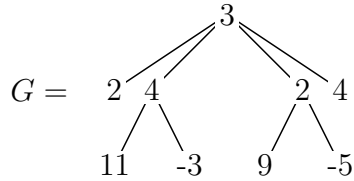

\begin{center}
\begin{graph}(4,2)

\roundnode{1}(1,0)\roundnode{2}(2,0)\roundnode{3}(3,0)\roundnode{4}(4,0)
\roundnode{5}(1,1)\roundnode{6}(1.5,1)\roundnode{7}(3.5,1)\roundnode{8}(4,1)
\roundnode{9}(2.5,2)

\edge{6}{1}\edge{6}{2}\edge{7}{3}\edge{7}{4}
\edge{9}{5}\edge{9}{6}\edge{9}{7}\edge{9}{8}

\nodetext{1}{11}\nodetext{2}{-3}\nodetext{3}{9}\nodetext{4}{-5}
\nodetext{5}{2}\nodetext{6}{4}\nodetext{7}{2}\nodetext{8}{4}
\nodetext{9}{3}

\freetext(0,1){$G=$}

\end{graph}
\end{center}
\caption{The impartial game $G=\{2,\{11|4|-3\}|3|,4,\{9|2|-5\}\}$}\label{ex1}
\end{figure}

\begin{theorem}
Impartial scoring play games form a non-trivial commutative monoid under the disjunctive sum.
\end{theorem}

\begin{proof}  A monoid is a semi-group that has an identity.  Scoring games in general have an identity, but it is a set with only one element, namely $\{.|0|.\}$, the proof of this is given in \cite{ES}.  We will show that if we restrict scoring play games to impartial games, then there is an identity set that contains more than one element.

First we will define a subset of the impartial games as follows;

$$I=\{i|G\plus i\approx G, \hbox{ for all impartial games }G\}$$  

To show that we have a non-trivial monoid we have to show that $I$ contains more than one element.  So consider the following impartial game, $i=\{\{0|0|0\}|0|\{0|0|0\}\}$.

\begin{figure}[htb]
\begin{center}
\begin{graph}(3,2)

\roundnode{1}(0,0)\roundnode{2}(0.5,1)\roundnode{3}(1,0)\roundnode{4}(1.5,2)
\roundnode{5}(2,0)\roundnode{6}(2.5,1)\roundnode{7}(3,0)

\edge{1}{2}\edge{2}{3}\edge{4}{2}\edge{4}{6}\edge{5}{6}\edge{6}{7}

\nodetext{1}{0}\nodetext{2}{0}\nodetext{3}{0}\nodetext{4}{0}
\nodetext{5}{0}\nodetext{6}{0}\nodetext{7}{0}

\end{graph}
\end{center}
\caption{The game $\{\{0|0|0\}|0|\{0|0|0\}\}$}
\end{figure}

To show that $i\plus G\approx G$ for all impartial games $G$, there are 3 cases to consider $\gfl>0$, $\gfl<0$ and $\gfl=0$, since the cases for Right follow by symmetry.  First let $\gfl>0$, if Left has no move on $G$, then neither does Right, since $G$ is impartial, i.e. $G=G^S$, so they will play $i$ and the final score will still be $G^S$.

So let Left have a move on $G$, Left will choose his best move on $G$.  Right has two choices, either continue to play $G$, or play $i$, and attempt to change the parity of $G$, i.e. force Left to make two consecutive moves on $G$.  However Left will simply respond by also playing $i$, and then it will be Right's turn to move on $G$ again.  Thus $(G\plus i)^{SL}_F>0$.

Next let $\gfl<0$, this means that no matter what Left does, he will lose playing $G$ on $G\plus i$, since Right will simply respond in $G$, until $G$ is finished, then they will play $i$, which does not change the final score of $G$.  Again if Left tries to change the parity of $G$, by playing $i$, Right will also play $i$, and it will be Left's turn to move on $G$ again.  Therefore $(G\plus i)^{SL}_F<0$.

Finally let $\gfl=0$.  This means that Left's best move will be a move that eventually ties $G$.  The only reason Left or Right would choose to move on $i$ is again to change the parity of $G$ and potentially win, i.e. forcing your opponent to move twice on $G$.  However if Left, say, moves on $i$ then Right will simply respond by also playing $i$ and it will be Left's turn to move on $G$ again, and similarly if Right moves on $i$, meaning that $(G\plus i)^{SL}_F=0$.

So therefore the set of impartial games is a non-trivial monoid and the theorem is proven.

\end{proof}

\begin{conjecture}
Not all impartial scoring play games have an inverse.
\end{conjecture}

To prove this one needs to show that given an impartial game $G$, for all impartial games $Y$ there is an impartial game $P$ such that $G\plus Y\plus P\not\approx P$.  This is very difficult to show, however it is extremely likely that this conjecture is true because for normal play games the inverse of any game $G$ is $-G$, and as we will now show there are impartial games $H$ where $-H$ is not the inverse.

So consider the game $G=\{2,\{1|2|3\}|0|-2,\{-3|-2|-1\}\}$, in this case $-G=G$.  If $G$ is the inverse of itself then $G\plus G\plus 0\approx 0$, in other words, $G\plus G\in \ti$.  However $G\plus G\in \pre$, this is easy to see since if Left moves first and moves to $2\plus G$, then Right can respond by moving to $2\plus \{-3|-2|-1\}$ and Left must move to $2\plus -3$ and loses.  If Left moves to $\{1|2|3\}\plus G$, then Right will move to $\{1|2|3\}\plus -2$ and Left must move to $1\plus -2$ and again loses.  Obviously the opposite will be true if Right moves first on $G\plus G$.  So $G\plus G\plus 0\not\approx 0$ and $G\plus G\not\in I$.

So, because $-G$ is not the inverse of $G$ in this case then it is very unlikely that any other impartial game could be $G$'s inverse, and while we do not have a proof of that, this simple example shows that it is probably true.

It is also worth noting that impartial scoring games can belong to any of the five outcomes for scoring games, i.e. $\lef, \ri, \pre, \n$ and $\ti$.  This is in stark contrast to both normal play and mis\'ere play games, where impartial games can only belong to either $\pre$ or $\n$.  

It is easy to see that this is true by considering an impartial game of the form $\{a|G^S|b\}$.  Clearly when $G^S=0$ then $b=-a$ and the outcome can only be $\n, \pre$ or $\ti$.  However we can set $G^S\neq 0$ and either large enough that both $a$ and $b$ are greater than zero, or less than zero, depending on if we make $G^S$ a very large positive or negative number.  In these cases the outcome will either be $\lef$ or $\ri$.

\section{Nim}

In the Introduction we gave the basic rules for nim played under normal and mis\`ere play.  For scoring play we will define the standard rules of nim a little differently;

\begin{enumerate}
\item{The initial score is $0$.}
\item{The game is played on heaps of beans, and on a players turn he may remove as many beans as he wishes from any one heap.}
\item{A player gets $1$ point for each bean he removes.}
\item{The player with the most points wins.}
\end{enumerate}

It should be clear that the best strategy for this game is simply to remove all the beans from the largest possible heap, and keep doing so until the game ends.  

Another thing to note is that, under normal play, for every single impartial game $G$ there is a nim heap of size $n$ such that $G=n$. This not the case with scoring play games, but as we will show in the next section, these games are still relatively easy to solve, regardless of the rules and of the scoring method.

\subsection{Scoring Sprague-Grundy Theory}\label{SGT}

Sprague-Grundy theory is a method that is used to solve any variation of a game of nim.  The function for normal play $\G(n)$ is defined in a such a way that if for a given heap $n$, played under some rules, if $\G(n)=m$ then this means that the original heap $n$ is equivalent to a nim heap of size $m$. 

For scoring play games this function is going to be defined slightly differently.  Rather than telling us equivalence classes of different games, it will tell us the final scores of games.  While this may not be as powerful as normal play Sprague-Grundy theory, it is still a very useful function and can be used to solve many different variations of scoring play nim.

One of the standard variations that have been used widely in books such as Winning Ways \cite{WW}, are a group of games called octal games.  These games cover a very large portion of nim variations, including all subtraction games.  For scoring games we will use the following definition;

\begin{definition}
A scoring play octal game $O=(n_1n_2\dots n_k, p_1p_2\dots p_k)$, is a set of rules for playing nim where if a player removes $i$ beans from a heap of size $n$ he gets $p_i$ points, $p_i\in \mathbb{R}$, and he must leave $a,b,c\dots$ or $j$ heaps, where $n_i=2^a+2^b+2^c+\dots+2^j$.
\end{definition}

By convention we will say that a nim heap $n\in O$ means that $n$ is played under the rule set $O$.  We will now define the function that will be the basis of our theory;

\begin{definition}
Let $n\in O=(t_1\dots t_f, p_1\dots p_f)$ and $m\in P=(s_1\dots s_e, q_1\dots q_e)$;

\begin{itemize}
\item{$\G_s(0)=0$.}
\item{$\G_s(n)=\max_{k,i}\{p_k-\G_s(n_1\plus n_2\plus \dots\plus n_{i})\}$, where $n_1+n_2+\dots +n_{i}=n-k$, $t_k=2^a+2^b+\dots 2^p$ and $i\in\{a,b,\dots, p\}$.}
\item{$\G_s(n\plus m)=\max_{k,i,l,j}\{p_k-\G_s(n_1\plus n_2\plus \dots\plus n_i\plus m),q_l-\G_s(n\plus m_1\plus m_2\plus\dots\plus m_j)\}$, where $n_1+n_2+\dots +n_i=n-k$, $t_k=2^a+2^b+\dots 2^p$ and $i\in\{a,b,\dots, p\}$, $m_1+m_2+\dots m_j=m-l$, $s_l=2^c+2^d+\dots +2^q$ and $j\in\{c,d,\dots, q\}$.}
\end{itemize}
\end{definition}

The first thing to prove is that this function gives us the information we want, namely the final score of a game.  So we have the following theorem;

\begin{theorem}\label{final}
$\G_s(n)=n^{SL}_F=-n^{SR}_F$ and $\G_s(n\plus m)=(n\plus m)^{SL}_F=-(n\plus m)^{SR}_F$.
\end{theorem}

\begin{proof}  The proof of this will be by induction on all heaps $n_1,n_2,\dots, n_i,m_1\dots,m_j$, such that $n_1+n_2\dots + n_i, m_1+\dots +m_j\leq K$ for some integer $K$, the base case is trivial since $\G_s(0\plus 0\plus 0\dots \plus 0)=0$ regardless of how many $0$'s there are.

So assume that the theorem holds for all $n_1,n_2,\dots, n_i,m_1\dots,m_j$, such that $n_1+n_2\dots + n_i, m_1+\dots +m_j\leq K$ for some integer $K$, and consider $\G_s(n\plus m)$, where $n+m=K+1$.

$\G_s(n\plus m)=\max_{k,i,l,j}\{p_k-\G_s(n_1\plus n_2\plus \dots\plus n_i\plus m),q_l-\G_s(n\plus m_1\plus m_2\plus\dots\plus m_j)\}$, but since $n_1+n_2\dots+n_i+m$ and $n+m_1+m_2\dots+m_j\leq K$, then by induction $\max_{k,i,l,j}\{p_k-\G_s(n_1\plus n_2\plus \dots\plus n_i\plus m),q_l-\G_s(n\plus m_1\plus m_2\plus\dots\plus m_j)\}= \max_{k,i,l,j}\{p_k-(n_1\plus n_2\dots\plus n_i\plus m)^{SL}_F, q_l-(n\plus m_1\plus\dots m_j)^{SL}_F\}=(n\plus m)^{SL}_F$, and the theorem is proven.
\end{proof}

\subsubsection{Subtraction Games}

Subtraction games are a very widely studied subset of octal games.  A subtraction game is a game of nim where there is a pre-defined set of integers and a player may only remove those numbers of beans from a heap.  This set is called a subtraction set.  From our definition of an octal game this means that each $n_i$ is either $0$ or $3$.  In this section we will also say that if a player removes $i$ beans then he gets $i$ points.

\begin{lemma}
Let $S$ be a finite subtraction set, then for all $s\in S$, $\G_s(s+2ik)=k-\G_s(s+(2i-1)k)$ for all $i\in\mathbb{N}$, where $k=\max\{S\}$.
\end{lemma}

\begin{proof}  We will split the proof of this into three parts;

\noindent\textbf{Part 1}: For all $i\in\mathbb{Z}^+$, $\G_s(r+2ik)\leq r$

The first thing to show is that for each $0\leq r\leq k$, $\G_s(r)\leq r$ and $\G_s(r+2ik)\leq r$ for all $i\in\mathbb{Z}^+$.  First let $r\leq k$, $\G_s(r)=\max_j\{j-\G_s(r-j)\}$ and since each $j$ in the set is less than or equal to $r$, and each $\G_s(r-j)\geq 0$, this implies that $\G_s(r)\leq r$.

Next let $\G_s(r+2ik)\leq r$ for smaller $i$, and consider $\G_s(r+2ik)=\max_j\{j-\G_s(r+2ik-j)\}$.  If $j\leq r$, then since $\G_s(r+2ik-j)\geq 0$, we have $j-\G_s(r+2ik-j)\leq j\leq r$.  If $j>r$, then $\G_s(r+2ik-j)=\G_s(r+k-j+(2i-1)k)\geq k-(r+k-j)=j-r$, by induction, therefore $j-\G_s(r+2ik-j)\leq j-(j-r)=r$.  So therefore $\G_s(r+2ik)\leq r$ for all $i$.

\noindent\textbf{Part 2}: For all $i\in\mathbb{Z}^+$, $\G_s(r+(2i+1)k)\geq k-r$

We also need to show that for each $0\leq r\leq k$, $\G_s(r+(2i+1)k)\geq k-r$ for all $i\in \mathbb{N}$.  Clearly $\G_s(r+k)\geq k-\G_s(r)\geq k-r$.  Again let $\G_s(r+(2i+1)k)\geq k-r$ for smaller $i$, then $\G_s(r+(2i+1)k)\geq k-\G_s(r+2ik)$ and from above we know that $\G_s(r+2ik)\leq r$ and hence $\G_s(r+(2i+1)k)\geq k-\G_s(r+2ik)\geq k-r$ for all $i$.

\noindent\textbf{Part 3}: For all $s\in S$ and $i\in\mathbb{Z}^+$, $\G_s(s+2ik)\geq s$ and $\G_s(s+(2i+1)k)\leq k-s$.

Let $s\in S$, then $\G_s(s)\geq s-\G_s(0)=s$, since we know from part 1 that $\G_s(s)\leq s$, this means that $\G_s(s)=s$.  So consider $\G_s(s+k)=\max_j\{j-\G_s(s+k-j)\}$, if $j\leq s$ then $j-\G_s(s+k-j)\leq j-k+\G(s-j)\leq j-k+s-j\leq s-k\leq k-s$.  If $j>s$ then $j-\G_s(s+k-j)\leq j-s+\G_s(k-j)\leq j-s+k-j=k-s$.  From part 2 we know that $\G_s(s+k)\geq k-\G_s(s)=k-s$, so $\G_s(s+k)=k-s$.

So assume that the theorem holds up to $i\geq 1$, and consider $\G_s(s+(2i+1)k)=\max_j\{j-\G_s(s+(2i+1)k-j)\}$.  If $j\leq s$ then  $j-\G_s(s+(2i+1)k-j)\leq j-k+\G_s(s+2ik-j)$, and from part 2 we know that $\G_s(s+2ik-j)\leq s-j$ therefore $j-k+\G_s(s+2ik-j)\leq j-k+s-j\leq s-k\leq k-s$.  

If $j>s$ then $j-\G_s(s+(2i+1)k-j)=j-\G_s(s+k+2ik-j)\leq j-s+\G_s(k-j+2ik)\leq j-s+k-j$, by induction, which is equal to $k-s$.

Finally consider $\G_s(s+(2i+2)k)\geq k-\G_s(s+(2i+1)k)$, and from before we know that $\G_s(s+(2i+1)k)\leq k-s$, therefore $k-\G_s(s+(2i+1)k)\geq k-(k-s)=s$.  So therefore $\G_s(s+(2i+2)k)=s$ and the lemma is proven.
\end{proof}

The obvious question to ask is does the lemma hold for all $n$?  The answer is no.  While it is clear that our function is eventually periodic for subtraction games at least, there are many examples where simply taking the largest number of beans, as in the lemma, is not always the best move.  For example consider a game with subtraction set $\{4,5\}$.  The table of this games $G_s(n)$ values are given in table \ref{o1}.

\begin{table}[htb]
\begin{center}
\begin{tabular}{|c|c|c|c|c|c|c|c|c|c|c|c|c|c|c|c|c|}
\hline
$n$&0&1&2&3&4&5&6&7&8&9&10&11&12&13&14&15\\\hline
$\G_s(n)$&0&0&0&0&4&5&5&5&5&1&0&0&0&3&4&5\\\hline
\end{tabular}
\end{center}
\caption{A game with subtraction set $\{4,5\}$.}\label{o1}
\end{table}

So in particular consider the value of $\G_s(13)$, this is $\max\{4-\G_s(9), 5-\G_s(8)\}=4-\G_s(9)=3$.  Therefore for this game taking 4 beans and gaining 4 points is preferable to taking 5 beans and gaining 5 points.  This is a very simple example to illustrate the point that we cannot say playing greedily would always work.  In other words we need to show that if $n$ is large enough then taking the largest number of beans available \emph{is} the best strategy.  So we make the following conjecture;

\begin{conjecture}
Let $S$ be a finite subtraction set, then there exists an $N$ such that $\G_s(n+2k)=\G_s(n)$ for all $n\geq N$, where $k=\max\{S\}$.
\end{conjecture}

It seems plausible that this conjecture is true, given the lemma, however it is also possible that there is an $n$ such that $\G_s(n+2ik)=J$ and $\G_s(n+(2i+1)k)=k-j$, where $J>j$.  What we have seen from the data is that often if $n\not\in S$ the values of $\G_s(n+2ik)$ and $\G_s(n+(2i+1)k)$ will alternate as in the lemma, but then you will reach an $i$ where the values change, and this switch might happen several times before it settles down.

A proof of the conjecture or a counter example would be a very big step forward in understanding how the function operates. 

\subsubsection{Taking-no-Breaking Games}

Taking-no-breaking games are a more general version of subtraction games, and cover a fairly wide range of octal games.  The rules of these games are fairly basic, when a player removes a certain number of beans from a heap, he will have one of three options.

\begin{enumerate}
\item{Leave a heap of size zero, i.e. remove the entire heap.}
\item{Leave a heap of size strictly greater than zero.}
\item{Leave a heap of size greater than or equal to zero.}
\end{enumerate}

From the definition of an octal game this means that each $n_i$ is either $0,1,2$ or $3$, also an octal game $O=(n_1n_2\dots n_k, p_1p_2\dots p_k)$ is finite if $k$ is finite.

It should be clear that for a fixed $m\in P$ and finite $O$, where $P$ and $O$ are two taking no breaking games, then the function $\G_s(n\plus m)$ must always be eventually periodic.  The reason is that we always compute each value from a finite number of previous values, and since $O$ is finite this implies that $\G_s(n\plus m)$ is bounded, and both of these facts together mean that the function will be eventually periodic.

The real question that one needs to answer however is not ``is it periodic?'', but ``what is the period?''.  We believe we can answer that question for a particular class of taking-no-breaking games, that is the class of games where if you remove $i$ beans you get $i$ points.  We make the following conjecture;

\begin{conjecture}\label{period}
Let $O=(n_1n_2\dots n_t, p_1p_2\dots p_t)$ and $P=(m_1m_2\dots m_l, q_1q_2\dots q_l)$ be two finite taking-no-breaking octal games such that, there is at least one $n_s\neq 0$ or $1$, and if $n_i$ and $m_j=1,2$ or $3$ then $p_i=i$ and $q_j=j$, and $p_i=q_j=0$, otherwise, then;

$$\G_s(n+2k\plus m)=\G_s(n\plus m)$$

\noindent where $O$ is finite and $k$ is the largest entry in $O$ such that $n_k\neq 0,1$.
\end{conjecture}

There is very strong evidence that this conjecture will hold.  Since $m$ is a constant it changes the value of $\G_s(n\plus m)$, but not the period.  We have checked the theorem for many examples and not yet found a counter example, which suggests that it is probably true.

Unfortunately proving it is surprisingly difficult.  The conjecture basically says that if $n$ is large enough, then your best move is to simply remove the maximum available beans from the heap $n$, so a proof would need to show that for any given $m$, there are only finitely many places where moving on $m$ or removing fewer than $k$ beans from $n$ is a better move.

There are several problems with this, the first is that the function $\G_s(n\plus m)$ only tells us the maximum possible value from the set of possible values.  This makes it very difficult to do a proof that first shows $\G_s(n+2k\plus m)\geq \G_s(n\plus m)$ and vice-versa.  The second is understand \emph{why} removing a lower number of beans would be better than playing greedily in some instances.

The last problem is induction is hard because what may hold for lower values may not hold at higher values, making a proof by induction difficult.  However since the function is recursively defined an inductive proof seems to be more natural than a deductive proof.

We believe that a proof of this theorem would also help in finding the period, and proving it for the more general case, where $i$ beans are worth $k$ points, $k\in \mathbb{R}$.

Of course it is natural to ask what happens in the general case, unfortunately in the general case the conjecture doesn't hold.  To see why consider the game $O=(3333, 2222)$.  The values of $G_s(n)$ are given in the following table;

\begin{table}[htb]
\begin{center}
\begin{tabular}{|c|c|c|c|c|c|c|c|c|c|c|c|}
\hline
$n$&0&1&2&3&4&5&6&7&8&9&10\\\hline
$\G_s(n)$&0&2&2&2&2&0&2&2&2&2&0\\\hline
\end{tabular}
\end{center}
\end{table}

This game has period 5, which does not correspond to a possible value of $k$, i.e. 1,2,3 or 4.    While all taking-no-breaking games are periodic as we can see from the example, it is not clear what the period is, since we can take our $p_i$'s to be any real number.  So we make the following conjecture;

\begin{conjecture}Let $O=(n_1n_2\dots n_t, p_1p_2\dots p_k)$ and $P=(m_1m_2\dots m_l, q_1q_2\dots q_l)$ be two finite taking-no-breaking octal games, then there exists a $t$ such that;

$$\G_s(n+t\plus m)=\G_s(n\plus m)$$

\end{conjecture}

\subsubsection{Taking and Breaking}

Another type of nim games we can examine are taking and breaking games.  That is games where after the player removes some beans from a heap he must break the remainder into two or more heaps. This is more general than taking-no-breaking games, since taking-no-breaking games are a subset of taking and breaking games.

There are several problems with examining taking and breaking scoring games.  The first is that we cannot even say that the function $\G_s(n\plus m)$ is bounded.  The reason is that with each iteration you are increasing the number of heaps, which may increase the value of the function as $n$ increases.  So we cannot put a bound on the function as we could with subtraction game and taking-no-breaking games.

Another problem is that if we were to say examine the game $0.26$, which means take one bean bean and leave one non-empty heap, or take two beans and leave either two non-empty heaps, or one non-empty heap, the number of computations required to find $\G_s(n)$ increases exponentially with $n$.  Since a heap of size $n-2$ may be broken into two smaller heaps $n_1$ and $n_2$, we must therefore also compute the value of $\G_s(n_1\plus n_2)$.  

However if $n_1-2=$ or $n_2-2$ may also be broken into two smaller heaps, say $n_1'$, $n_1''$, $n_2'$ and $n_2''$ then we must compute the value of $\G_s(n_1'\plus n_1''\plus n_2)$ and $\G_s(n_1\plus n_2'\plus n_2'')$.  This process will continue until we have heaps that are too small to be broken up.  So this means that computing $\G_s(n)$ for a taking and breaking game is a lot harder, than for a taking-no-breaking game, simply due to the number of computations involved.

So we have the following conjecture.

\begin{conjecture}
Let $O=(n_1n_2\dots n_k, p_1p_2\dots p_k)$ and $P=(m_1m_2\dots m_l, q_1q_2\dots q_l)$ be two finite octal games then there exists a $t$ such that;

$$\G_s(n+t\plus m)=\G_s(n\plus m)$$

\end{conjecture}

While we feel that this conjecture may be true, it is certainly not as strong a conjecture as conjecture \ref{period}, for the reasons previously given.  However studying these games would certainly be interesting and anything anyone could find out about them would be useful.

\section{Conclusion}

We hope that we have given the readers some interesting new ideas about the types of games that can be studied with scoring play theory, as well as opening up a whole new world of impartial games that can be researched.  We have simply introduced the ideas, but there is still much to be learned from these fascinating games.

\section{Acknowledgements}

I would like to thank my supervisor Keith Edwards for giving me a lot of help with this paper, particularly in helping me gather a lot of data about the function $\G_s(n)$ and proving lemma 11.


\begin{thebibliography}{9}

\bibitem{WW}{E. Berlekamp, J. Conway, R. Guy, \textit{Winning Ways for your
Mathematical Plays}, Volumes 1-4, A.K. Peters 2002}

\bibitem{ONAG}{J. Conway, \textit{On Numbers and Games}, A.K. Peters
2000}

\bibitem{PG}{P. Grundy, \textit{Mathematics and Games}, Eureka, \textbf{2}(1939) 6-8; reprinted ibid.\textbf{27}(1964) 9-11.}

\bibitem{PS}{P. Grundy, C. Smith, \textit{Disjunctive games with the last player losing} Proc. Camb. Philos, Soc., \textbf{52}(1956), 443-458.}

\bibitem{TP}{T. Plambeck, \textit{Taming the Wild in Impartial Combinatorial Games}, Integers, Volume 5(1)(2005).}

\bibitem{RS}{R. Sprague, \textit{Uber Mathematische Kampfspiele}, Tohoku Math. J., \textbf{41}(1935-1936) 438-444; Zbl.\textbf{13}, 290.}

\bibitem{ES}{F. Stewart, \textit{Scoring Play Combinatorial Games}, Submitted to Games of No Chance 5.}

\bibitem{FS}{F. Stewart, \textit{Scoring Play Combinatorial Games}, PhD Thesis, University of Dundee, October 2011.}

\end{thebibliography}
\end{document}